\documentclass[11pt]{amsart}
\usepackage{amsfonts}
\usepackage{amscd}
\usepackage{amsmath}
\usepackage{epsfig}

\newlength{\dinwidth}
\newlength{\dinmargin}
\newtheorem{definition}{Definition}
\newtheorem{theorem}{Theorem}

\newtheorem{remark}{Remark}

\def\be{\begin{equation}}
\def\ee{\end{equation}}
\def\ben{\begin{displaymath}}
\def\een{\end{displaymath}}
\def\baa{\begin{eqnarray}}
\def\eaa{\end{eqnarray}}

\def\ba{\begin{array}}
\def\ea{\end{array}}
\makeatletter \@addtoreset{equation}{section} \makeatother

\newtheorem{thm}{Theorem}[section]

\newtheorem{coro}[thm]{Corollary}
\newtheorem{lem}[thm]{Lemma}
\newtheorem{prop}[thm]{Proposition}

\newtheorem{defn}[thm]{Definition}
\newtheorem{eg}[thm]{Example}

\newtheorem{remk}[thm]{Remark}

\newcommand{\supp}{{\rm supp}}

\newcommand{\xit}{{\tilde{\xi}}}

\newcommand{\eps}{\varepsilon}
\newcommand{\bet}{\beta}
\newcommand{\spec}{\rm spec}
\newcommand{\alp}{\alpha}

\newcommand{\R}{{\mathbb{R}}}
\newcommand{\B}{{\mathbb{B}}}

\newcommand{\C}{{\mathbb{C}}}

\newcommand{\Z}{{\mathbb{Z}}}

\newcommand{\Cc}{{\mathcal C}_0}

\newcommand{\refeq}[1]{(\ref{#1})}

\newcommand{\dL}{\Delta_L}
\newcommand{\dF}{\Delta_F}
\newcommand{\la}{\lambda}
\newcommand{\lat}{{\tilde{\lambda}}}

\newcommand{\La}{\Lambda}

\newcommand{\Ccinf}{{\mathcal C}_0^\infty}
\newcommand{\dom}{\mbox{dom}}
\newcommand{\Tr}{{\rm Tr}}
\newcommand{\Green}{\mathcal{G}}
\newcommand{\dis}{\displaystyle}
\newcommand{\und}{{\frac{1}{2}}}

\renewcommand{\Re}{\mbox{Re}}

\title{Krein formula and $S$-matrix for Euclidean surfaces with conical singularities}

\author[L. Hillairet]
{Luc Hillairet}
\email{Luc.Hillairet@math.univ-nantes.fr}
\address{UMR CNRS 6629-Universit\'{e} de Nantes, 2 rue de la Houssini\`{e}re, \\
BP 92 208, F-44 322 Nantes Cedex 3, France}

\author[A. Kokotov]
{Alexey Kokotov}
\email{alexey@mathstat.concordia.ca}
\address{Department of Mathematics and Statistics, Concordia University\\
1455 de Maisonneuve Blvd. West \\
Montreal, Quebec H3G 1M8 Canada}

\begin{document}

\begin{abstract}
Using the Krein formula for the difference of the resolvents of two self-adjoint extensions of a symmetric operator with finite deficiency indices,
we establish a comparison formula for $\zeta$-regularized determinants of two self-adjoint extensions of the Laplace operator on a Euclidean surface with conical singularities (E. s. c. s.). The ratio of two determinants is expressed through the value $S(0)$ of the $S$-matrix, $S(\lambda)$, of the surface.
 We study the asymptotic behavior of the $S$-matrix, give an explicit expression for $S(0)$ relating it to the Bergman projective connection on
 the underlying compact Riemann surface and derive variational formulas for $S(\lambda)$ with respect to coordinates on the moduli space of E. s. c. s. with trivial holonomy.
\end{abstract}

\maketitle

\section{Introduction}
Spectral geometry aims at understanding the relations between the spectrum
of some Laplace operator in a given geometrical setting and  geometric
properties of the latter. Polygons and polyhedra are among the simplest shapes
one can consider and one could hope in this setting for a better understanding.
This leads naturally to study the spectral geometry of
Euclidean surfaces with conical singularities. Another motivation is the
spectral theory of translation surfaces for which the geometric
picture has many interesting developments (see \cite{Zorich} for instance).

One peculiarity of Laplacians on manifolds with conical points is that,
due to the presence of conical points,
a choice has to be made in order to get a self-adjoint
operator. In this paper, we are interested in understanding how this
choice affects several spectral
quantities such as the resolvent and the zeta-regularized
determinant. Depending on the self-adjoint extension, this
zeta-regularization procedure is not as straightforward as usual because
of unusual behavior of the zeta function but it is still
possible to define such a regularization
(see \cite{KLP, GKM} and section \ref{sec:zetareg}) and  we will prove a comparison
formula for these determinants.

Comparison formulas for regularized determinants for conical manifolds were first found
in \cite{LMP} using  a surgery formula
{\em \`a la} BFK (see \cite{BFK}) and
in \cite{KLP} using a contour integral method based on a secular equation
that defines the spectrum. One of our motivations was to understand how
the comparison formulas for different self-adjoint extensions
from \cite{LMP} read in the case
of Euclidean surfaces with conical singularities and whether it is
possible to express the determinants of the non Friedrichs
self-adjoint extensions of the Laplacian on these surfaces through
holomorphic invariants of the underlying Riemann surface (as it was
done in \cite{KokKor} for the determinant of the Friedrichs
extension). Indeed, Euclidean surfaces with conical singularities
are our primary interest and we will restrict to this setting
although many statements still make sense for more general conical
manifolds.

It turns out that the geometric interpretation of the formulas obtained in
\cite{LMP} and \cite{KLP} is not that straightforward and we have found
it more convenient to establish the comparison formula for determinants
using the Krein formula for the difference of
resolvents of two self-adjoint extensions of a symmetric operator.
We observe that the trace of the difference of two resolvents admits a nice representation through
the so-called $S$-matrix of a Euclidean surface with conical singularities (E. s. c. s.) $X$.
The latter matrix, or, more precisely, the meromorphic family of matrices $S(\lambda)$
is in some sense a characteristic feature of $X$. Indeed, we believe that some of the geometry
of $X$ (such as for instance the lengths of saddle-connections between conical points- see Remark \ref{rk:difforbit})
is encoded in $S(\lambda)$ although it seems quite difficult to retrieve this kind of information.
We should also remark that this
$S$-matrix allows to write down a secular equation that can then
be treated using the approach of \cite{KLP} so that what we
propose here may be seen as a geometric interpretation for the
latter method. The comparison with \cite{LMP} is less straightforward,
it relies in interpreting the $S$-matrix as some kind of limiting
Dirichlet-to-Neumann operator on a circle around the conical point
when the radius of that circle goes to $0.$ It can be noted here that,
in contrast with \cite{LMP} no extra condition is needed to obtain
our formula.

We will thus prove the following theorem. The notion of {\em regular} self-adjoint extensions will be
introduced in definition \ref{def:reg}, section \ref{sec:Kfrd} and, for these self-adjoint extensions,
the expression $P+QS(0)$ makes sense (see remark \ref{rk:DregL}).

\begin{theorem}\label{thm:intromain}
On a compact E.s.c.s. X, let $S(\la)$ be the $S$-matrix and $\dF$ be the Friedrichs
extension.

Let $P$ and $Q$ be matrices that define a regular self-adjoint extension $\dL,$
and define \[ D(\la):=\det(P+QS(\la)).\]
Let $d$ be the dimension of $\ker(P+QS(0))$ and let $D^*(0):= \lim_{\la\rightarrow 0} (-\la)^{-d}D(\la).$

There exist $\alp_0$ and $\Gamma$ such that the asymptotic expansion of $D(-|\lambda|)$ as $\lambda$ goes to $\infty$ is
\[
\ln D(-|\lambda|):=\alp_0\ln(|\lambda|)+\Gamma +o(1).
\]
The following identity then holds :
\[
{\det}^*_{\zeta}(\dL)=\exp(-\Gamma)D^*(0){\det}^*_\zeta(\dF),
\]
in which ${\det}^*_{\zeta}$ is the modified zeta-regularized determinant
(see definition \ref{def:moddet}.)
\end{theorem}

To fulfil our second aim we then need to understand more explicitly
what kind of geometric information is encoded
in the family $S(\la).$  We focus on the limiting behavior
when the spectral parameter goes to $0$ since this is the regime that
comes up in the comparison formula. We will prove
that most of the matrix elements in this limit have an
interpretation through values of the Bergman projective connection
and the basic holomorphic differentials taken at the conical point
in the corresponding distinguished holomorphic local parameter (see section \ref{sec:S0}).
Since we expect translation surfaces to have particular and interesting
features, we will also say a word on the $S$-matrix on these
special kind of surfaces. Namely, we will derive variational formulas for the
$S$-matrix when it is differentiated with respect to moduli parameters.
These results answer most of the questions which motivated our study.

\subsection*{Organization of the paper}
In the small second section we will recall the basic facts about Euclidean surfaces with
conical singularities. We will in particular recall that
these can be viewed as Riemann surfaces with flat conformal conical metric.

In  section 3 we recall some basic properties of the Friedrichs
Laplace operator on E. s. c. s.,
and introduce the object of our primary interest --- the $S$-matrix;  we also derive
here the standard formula for the
derivative of the $S$-matrix with respect to $\la$.

In section 4 we study the asymptotic behavior of $S(\la)$ as  $\la$
goes to $-\infty$ and find the geometric interpretation of $S(0).$ We also
also apply the variational formulas of \cite{KokKor} to obtain the variations of
$S(\la)$ with respect to moduli parameters on translation surfaces.

In section 5 we study various self-adjoint extensions of the Laplace operator on E. s. c. s. and prove the comparison formula for their
$\zeta$-regularized determinants.

{\bf Acknowledgements.}

The research of LH was partly supported by the ANR programs {\em NONaa} and
{\em Teichm\"uller}.

The research of AK was supported by NSERC.
AK thanks Hausdorff Research Institute for Mathematics (Bonn) and Laboratoire de Math\'ematiques Jean Leray (Nantes) for hospitality.
AK also thanks the MATPYL program for supporting his coming and stay in Nantes where this research began.

 We acknowledge useful conversations with G. Carron and with D. Korotkin whose advice in particular helped us to simplify some constructions
from section 4.2.

\section{Euclidean surfaces with conical singularities}
\subsection{Euclidean surfaces with conical singularities as Riemann surfaces
with conformal flat conical metrics}
A Euclidean surface with conical singularities (E. s. c. s.) is a compact (orientable)
surface glued from Euclidean triangles. One can take as an example of such a surface
the boundary of a connected but not necessarily simply connected polyhedron in ${\mathbb R}^3$.

When two triangles are glued together and after rotating one of the triangles
around the common edge we observe that the intrinsic geometry of the surface is locally that of the plane. 
There, the surface actually is smooth and equipped with a smooth Euclidean metric.
At a vertex $p$ where $k$ triangles with angles $\vartheta_1, \dots, \vartheta_k$ are glued together,
the surface is locally isometric to a neighbourhood of the tip of the Euclidean cone
of total angle $\theta_p= \vartheta_1+\dots+\vartheta_k.$ The surface $X$ is thus equipped with
a Euclidean metric that is smooth except at the vertices $p$ for which $\theta_p \neq 2\pi$.

It follows for instance from \cite{Troy} that $X$ can be
provided with a complex analytic structure becoming a compact Riemann surface $\tilde{X}$;
moreover, the usual Euclidean metric on $X$ gives rise to a flat {\it conformal}
(i. e. defining the same complex structure) metric on $\tilde{X}$. Abusing notations slightly, from
now on we won't make any difference between $X$ and $\tilde{X}$.

On the other hand, consider a flat conformal metric $m$ with conical singularities
on a Riemann surface $X.$ In a vicinity of a conical point $p,$ $m$
can be written as
\[
m=|g(z)||z|^{2b}|dz|^2,\]
where $z$ is a holomorphic local parameter near $p$,
$z(p)=0$,  $b>-1$ and $g(z)$ is a holomorphic function of the local
parameter such that $g(0)\neq 0$.

It is shown in \cite{Troy} that one can choose a holomorphic change of variables $z=z(\zeta)$ such that
\begin{equation}\label{dp}|g(z(\zeta))||z(\zeta)|^{2b}|z'(\zeta)|^2=|\zeta|^{2b}\end{equation}
and, therefore,
\begin{equation}\label{dpmetric}
m=|\zeta|^{2b}|d\zeta|^2
\end{equation}

in the local parameter $\zeta$. This means that the Riemannian surface $(X, m)$ near $p$ is isometric
to the standard Euclidean cone of angle $2\pi(b+1)$. Troyanov \cite{Troy} showed that
the Riemannian manifold $(X, m)$ can be triangulated in such a way that all the conical points
will be among the vertices of the triangulation meaning thus that $(X, m)$ is an E. s. c. s.

\begin{definition} Let $X$ be a compact Riemann surface with conformal
flat conical metric (i. e. a  E. s. c. s.) and let $p\in X$ be a conical point.
Then any holomorphic local parameter $\zeta$ in which the metric takes the form (\ref{dpmetric}) is called distinguished.
\end{definition}

\noindent{\em Notation :} We will denote by $P$ the set of conical points and by $X_0:=X\setminus P$ the complement
of $P$ in $X.$ We set $M:= \mathrm{Card}(P)$ the number of conical points.
At each $p\in P,$ the total cone angle is denoted by $\theta_p.$

\subsubsection{Translation and half-translation surfaces.}
A translation (resp. half-translation) surface is a E. s. c. s. that has
trivial holonomy (resp.  holonomy group $\Z_2$). These are important
examples of E.s.c.s. with very nice geometric properties
(see \cite{Zorich} for a survey on these).

Translation surfaces are Riemann surfaces $X$ that are equipped with a
conformal flat conical metric given by the modulus square, $m=|\omega|^2$,
of a holomorphic $1$-form (an Abelian differential) $\omega$.
If $P$ is a zero of $\omega$ of multiplicity $k$
then $p$ is a conical point of the translation surface $X$ with conical angle $2\pi(k+1)$.
The moduli space $H_g$ of pairs $(X, \omega)$ (where $X$ is a compact Riemann surface
of genus $g\geq 1$, $\omega$ is a holomorphic $1$-form on $X$) is stratified according
to the multiplicities of the zeros of the $1$-form $\omega$.
Denote by $H_g(k_1, \dots, k_M)$ the stratum consisting of pairs $(X, \omega)$,
where $\omega$ has $M$ zeros, $p_1, \dots, p_M$ of multiplicities $k_1, \dots, k_M$
(according to Riemann-Roch theorem one has $k_1+\dots+k_M=2g-2$).
The stratum $H_g(k_1, \dots, k_M)$ is a complex orbifold of dimension $2g+M-1$.

Let $(X, \omega)\in H_g(k_1, \dots, k_M)$.
Choose a canonical basis of cycles $\{a_\alpha, b_\alpha\}$ on the Riemann surface $X$ and
take $M-1$ contours $\gamma_k$, $k=2, \dots, M$ on $X$
connecting $p_1$ with $p_2$, \dots, $p_M$

The local coordinates on $H_g(k_1, \dots, k_M)$ (which are called Kontsevich-Zorich
homological coordinates, see \cite{KZ}) are given by
the following integrals:
$$A_\alpha=\oint_{a_{\alpha}}\omega;\ \ \ \  \alpha=1, \dots, g,$$
$$B_\alpha=\oint_{b_{\alpha}}\omega;\ \ \ \  \alpha=1, \dots, g,$$
$$z_k=\int_{\gamma_k}\omega; \ \ \  \ k=2, \dots, M-1\, .$$

A half-translation surface is a compact Riemann
surface with flat conical metric $m=|q|$, where $q$ is a meromorphic quadratic
differential with at most simple poles.

\begin{eg}\label{eg:spherepant}
Consider the Riemann sphere ${\mathbb C}P^1$
with  metric
$$\frac{|z|^2|dz|^2}{\prod_{k=1}^6|z-z_k|}\,,$$
where $z_k\in {\mathbb C}, z_k\neq 0$ and $z_i\neq z_k$ if $i\neq k$.
This is a half-translation surface with $7$ conical points $0, z_1, \dots, z_6$.
The conical angle at $0$ is $4\pi$, the conical angles at each point $z_k$ are equal
to $\pi$.

Such a surface can be viewed by considering a Euclidean
pair of pants (with one $4\pi$ singularity) and by sewing each leg and the waist
with itself (thus creating the six $\pi$ singularities).
\end{eg}

\section{The Friedrichs Laplacian and the S-matrix}\label{sec:defFried}
Let $X$ be a compact E.s.c.s.. In this section we will recall the definition of
the Friedrichs Laplacian associated with the (singular) metric and define
the so-called $S$-matrix. We will then collect several properties of this matrix.

We denote by $\Delta$ the minimal closed extension of the
Euclidean Laplacian defined on $\Ccinf(X_0),$ and by
$\Delta^*$ its adjoint with respect to the Euclidean $L^2$ scalar product
\[
\langle u, v\rangle := \int_X u\overline{v} \,dx.
\]

Near each conical point $p,$ any $u\in \dom(\Delta^*)$ has the following asymptotic behavior in polar
coordinates $(r,\theta)$~(see, e. g., \cite{Mooers}, \cite{LMP}, \cite{NazPla} or \cite{KokKor})~:
\begin{equation}\label{defab}
u(r,\theta)\,=\, \sqrt{2\theta_p}\left(a^+_0 + a^-_0 \ln(r)\right) + \sum_{\nu} \sqrt{2|\nu|\theta_p}\left(a^+_\nu r^{|\nu|}+a^-_\nu r^{-|\nu|}\right)
\exp(i\nu \theta) + u_0,
\end{equation}
where $\nu$ ranges over $N_p:= \left \{ \frac{2\pi}{\theta_p}\cdot k,~ |~k\in\Z\backslash \{0\},~|k|< \frac{\theta_p}{2\pi} \right \},$
and $u_0\in \dom(\Delta)$.

 \hfill \\
\noindent{\em{Notation :}} We will denote by $N=\cup_{p\in P} N_p,$ and we will abusively still denote
by $\nu$ an element of $N.$ Choosing an element $\nu$ of $N$ thus amounts to
choosing a conical point $p$ and then some $\nu$ in $N_p.$ Unless needed we will omit the
reference to $p.$ The square roots prefactor in (\ref{defab}) are just normalization constants.
We will denote these constants by $C_0:=\sqrt{2\theta_p}$ and $C_\nu:=\sqrt{2|\nu|\theta_p}$
(we recall that since $\nu$ implicitly depends on $p$, so does $C_\nu$).

In the distinguished local parameter $\zeta$ near $p$ we have, for
$\nu = \frac{2\pi}{\theta_p}\cdot k$
\begin{gather}
\zeta^k\,=\,r^{\nu} \exp(i\nu \theta) \,=\, \left\{ \begin{array}{lr}  r^{|\nu|} \exp(i\nu \theta) & \mbox{if}~~\nu >0,\\
\\
								  r^{-|\nu|} \exp(i\nu \theta) & \mbox{if} ~~\nu<0.
					\end{array} \right .
\\
\overline{\zeta}^{-k}\,=\,r^{-\nu} \exp(i\nu \theta) \,=\,
\left\{ \begin{array}{lr}  r^{|\nu|} \exp(i\nu \theta) & \mbox{if}~~\nu < 0,\\
\\
r^{-|\nu|} \exp(i\nu \theta) & \mbox{if} ~~\nu>0. \end{array} \right .
\end{gather}					

Thus the asymptotic expansion \refeq{defab} may also be written
\begin{equation}\label{eq:defabdlp}
u(\zeta,\bar{\zeta})\,=\, C_0\left(a^+_0 + a^-_0 \ln(|\zeta|)\right) + \sum_{k=1}^{\frac{\theta_p}{2\pi}-1} C_{k\frac{2\pi}{\theta_p}}\left(a^+_k \zeta^k+a^-_k \bar{\zeta}^{-k} +
a^+_{-k}\bar{\zeta}^k\,+\,a^-_{-k}\zeta^{-k}\right )
 + u_0.
\end{equation}
	
A straightforward application of Green's formula (combined with the choice of the normalization constants $C_0,\, C_\nu$)
then implies that, for any $u,v$ in
$\dom(\Delta^*)$,
\begin{gather}\label{Green1}
\langle \Delta^* u, v\rangle\, -\,\langle u, \Delta^* v\rangle \,=\,
\sum_{p\in P}\left [ {a^+_0}\cdot \overline{b^-_0}\,-\,
{a^-_0} \cdot \overline{b^+_0}\,+\,\sum_{\nu \in N_p}  \left({a^+_\nu}
\cdot \overline{b^-_\nu}\,-\,
{a^-_\nu} \cdot \overline{b^+_\nu}\right)\right] \,
\end{gather}
where the $a^\pm_\nu$ are the coefficients in the expansion of $u$ and the
$b^\pm_\nu$ those in the expansion of $v.$

Setting $\Green(u,v):= \langle \Delta^* u, v\rangle\, -\,\langle u, \Delta^* v\rangle$ we define a Hermitian symplectic form
on $\dom(\Delta^*)/\dom(\Delta)$ whose lagrangian subspaces parametrize the
self-adjoint extensions of $\Delta.$

\subsection{The Friedrichs extension}
For any $u\in \dom(\Delta)$ a straightforward integration by parts gives
\[
\langle \Delta u, u \rangle = \int_X |\nabla u|^2 dx
\]
so that the Friedrichs procedure (see \cite{BirSol} section 10.3 or \cite{RSII} theorem X.23)
provides us with a self-adjoint extension that we denote by $\Delta_F.$
Since a function $u$ in $\dom(\Delta_F)$ is characterized by $\nabla u \in L^2(X),$ we obtain the
following lemma.

\begin{lem}
The lagrangian subspace in $\dom(\Delta^*)/\dom(\Delta)$ that corresponds to
the Friedrichs extension is
\[
\left \{ a_\nu^- =0 \right \}.
\]
\end{lem}

\begin{definition}
We denote by $H^s:=\dom( \dF^{\frac{s}{2}})$ the scale of Sobolev spaces associated with it.
In particular we set $\dom(\Delta_F):= H^2.$
\end{definition}

\begin{remark}
This definition of $H^s$ is not completely standard. In particular, because of the conical
singularities, for $m>1$ the following inclusion is strict (see \cite{Forni} for a much more detailed
discussion about this fact) :
\[
\{ u\in L^2~|~\forall |\alp|\leq m,  \partial^{\alp} u\in L^2\} \subset H^m.
\]
\end{remark}

By standard spectral theory, the resolvent of $\Delta_F$ defines a continuous operator from $H^s$ to $H^{s+2}$.
We also recall that since $X$ is compact, the Rellich-type injection theorem
from \cite{CT} implies that $\dF$ has compact resolvent so that the spectrum is non-negative
and discrete.

\subsection{The S-matrix}
We will now define a matrix associated to the flat structure and to the choice of the
Friedrichs extension.

First, for any $\nu,$ we fix  $F_\nu\,=\,C_\nu r^{-|\nu|}\exp(i\nu\theta)\rho(r)$ where
$\rho$ is some fixed cut-off function that is identically $1$ near the corresponding
conical point $p.$

We define $\Lambda_\nu$ to be the linear functional on $H^2$ satisfying
\begin{equation}\label{eq:Green2}
\forall u\in H^2,~~\Lambda_\nu(u)\,=\,\Green(u,F_\nu).
\end{equation}

We have the following lemma.
\begin{lem}
The linear functional $\Lambda_\nu$ is continuous on $H^2$ and
\[
\forall u\in H^2,~~\Lambda_\nu(u)\,=\,a^+_\nu
\]

 where $a^+_\nu$ is the coefficient in the
expansion \refeq{defab} of $u$ near $p.$
\end{lem}

\begin{proof}
The fact that $F_\nu \in \dom(\Delta^*)$ implies that
$\Lambda_\nu$ is indeed continuous. The second statement follows from
the respective asymptotic behaviors of $F_\nu$ and $u$ near $p$.
\end{proof}

\begin{remk}
The preceding lemma in particular implies that the linear functional $\Lambda_\nu$
doesn't depend on the choice of the cut-off function $\rho.$
\end{remk}

For $\lambda \in \C\setminus [0,\infty),$ we set
\[
G_\nu(\cdot\,;\la) := \, (\Delta_F-\lambda)^{-1} \Lambda_\nu.
\]

Since $\Lambda_\nu$ is in $H^{-2}$, $G_\nu$ is in $L^2,$ and for any $u\in H^2$, we have
\begin{equation}\label{eq:LambdaG}
\Lambda_\nu(u)\,=\, \langle (\Delta_F-\lambda)u, G_\nu(\cdot\,;\overline{\la}) \rangle .
\end{equation}

Since the resolvent is analytic in $\la$, $G_\nu(\cdot\,;\la)$ defines an analytic family of $L^2$
functions.

Observe that the latter equation is equivalent to
\[ (\Delta^*-\lambda)G_\nu(\cdot;\lambda) =0, \]
so that $G_\nu(\cdot; \lambda)\in \dom(\Delta^*).$
Moreover, by testing against an appropriate $u\in H^2$ we
can compute the coefficients $a^-_\mu$ of $G_\nu.$
This yields  $a^-_\mu= \delta_{\mu\nu}$ (where $\delta$ is the Kronecker symbol).

The following proposition gives a formula for $G_\nu$

\begin{prop}\hfill\\
For any $\lambda \in \C\setminus [0,\infty),$ set $f_\nu(\cdot\, ;\lambda):= (\Delta^*-\lambda)F_\nu$ and
$g_\nu(\cdot\,;\lambda)\,:=\, -\left ( \Delta_F -\lambda\right )^{-1}f_\nu(\cdot ; \lambda).$
Then $g_\nu(\cdot\,;\la)$ is an analytic family in $H^2$ and
\[
G_\nu(\cdot\, ; \lambda)\,=\, F_\nu(\cdot) \,+\,g_\nu(\cdot\, ; \lambda).
\]
\end{prop}

\begin{proof}
Computation shows that $f_\nu$ is in $L^2(X)$ which yields that
$g_\nu$ is in $H^2$ since $\lambda$ is in the resolvent set of $\Delta_F.$
Since $f_\nu$ and the resolvent depend analytically on $\la$ so does $g_\nu.$
By construction, $\left(\Delta^*-\lambda\right)(F_\nu + g_\nu)=0$
and all the $a_\mu^-$ coefficients of $G_\nu - (F_\nu+g_\nu)$ vanish.
This means that the latter function is in $H^2$ and thus is $0$ since
$\lambda$ is in the resolvent set.
\end{proof}

\begin{eg}\label{GonCone}
Let us consider the complete cone $[0,\infty)\times \R/\alp\Z.$
Using separation of variables we have that $G_\nu(r,\theta\,;\lambda)\,=\,k(r)\exp(i\nu\theta).$
For $\nu\neq 0$, by
definition $k$ is the unique solution to
\[
-k''-\frac{1}{r}k'+\left(\frac{\nu^2}{r^2}-\la \right)k\,=\,0,
\]
which is $L^2(rdr)$ and asymptotic to $C_\nu r^{-|\nu|}$ near $0.$
Thus $k$ is proportional to $K_\nu(\sqrt{-\lambda}\,r)$ where
$K_\nu$ is Bessel-MacDonald function (see \cite{Olver} for instance).
For $\nu=0,$ the singular behavior is logarithmic but
$k(r)$ still is proportional to $K_0(\sqrt{-\la}\,r)$
\end{eg}

\begin{defn}[The S-matrix]
We define the S-matrix $S(\lambda)$ by
\begin{equation}\label{defS}
S_{\mu\nu}(\la)\,=\,\Lambda_\mu(g_\nu(\cdot\,;\la)).
\end{equation}
\end{defn}
\hfill \\

\begin{remk}
Alternatively, $S_{\mu\nu}(\la)$ is the $a^+_\mu$ coefficient of $g_\nu(\cdot\,;~\la)$.
It is also the $a^+_\mu$ coefficient of $G_\nu(\cdot\,;~\la).$
Observe that the entries of the $S$-matrix are numbered by non-integer numbers.
\end{remk}

Using \refeq{eq:LambdaG}, we have the following alternative expression
\begin{equation*}
S_{\mu\nu}(\la)\,=\,\left\langle \left( \Delta_F-\la\right)g_\nu(\cdot\,;\la),G_\mu(\cdot\,;\overline{\la})\right\rangle \,
=\,\langle f_\nu(\cdot\,;\la),G_\mu(\cdot\,;\overline{\la})\rangle
\end{equation*}

It follows from the analyticity of $g_\nu$ that $S(\lambda)$ is analytic on $\C\setminus [0;\infty).$

\begin{eg}\label{SonCone}
We define $S_\alp(\la)$ to be the S-matrix of the cone of angle $\alp.$
According to example \ref{GonCone}, $S_\alp(\la)$ is diagonal.
Moreover, the asymptotic expansion of Bessel-Macdonald functions near $0$ is
\begin{gather*}
K_0(z)\,=\,-\ln (z)+\ln(2)-\gamma\,+\,o(1),\\
K_{|\nu|}(z)\,=\,\frac{\pi}{2\sin(|\nu|\pi)}\left[ \frac{z^{-|\nu|}}{2^{-|\nu|}\Gamma(1-|\nu|)}
-\frac{z^{|\nu|}}{2^{|\nu|}\Gamma(1+|\nu|)}+O(z^{2-|\nu|})\right ]
\end{gather*}
where $\Gamma$ is Euler gamma function and $\gamma$ Euler's constant (see for instance \cite{Olver}).
This yields
\begin{gather*}
[S_\alp(\la)]_{00}\,=\, \ln(\sqrt{-\la})-\left( \ln(2)-\gamma \right),\\
\\
[S_\alp(\la)]_{\nu\nu}\,=\,-\frac{\Gamma(1-|\nu|)(-\la)^{|\nu|}}{2^{2|\nu|}\Gamma(1+|\nu|)}.
\end{gather*}
\end{eg}

The interpretation of $S(\lambda)$ is given by the following lemma.

\begin{lem}
For any $\la \in \C\setminus [0,\infty)$ and any $F\in \ker(\Delta^*-\la)$.
Denote by $A^{\pm}(F)$ the vector consisting of all the coefficients $a^-_\nu$ (resp.
$a^+_\nu$) of $F.$
Then we have
\[
 A^+\,=\, S(\la)A^-.
 \]
\end{lem}

\begin{remk}
Interpreting $A^-$ as some kind of incoming data and
$A^+$ as the outgoing data justifies the interpretation of the $S$-matrix as
a scattering matrix.
\end{remk}

\begin{proof}
Set $\tilde{F}:= \sum_{\nu} a^-_\nu G_\nu(\cdot\,;\la)$ then $F-\tilde{F}$ is in
$\dom(\ker(\Delta^*-\la)).$ Since all the $a^-_\nu$ vanish, $F-\tilde{F}$ actually is in
$\dom(\dF)$. This implies $F=\tilde{F}$ since $\la$ is in the resolvent set of $\dF.$
Writing each $G_\nu\,=\,F_\nu+g_\nu$, we obtain :
\[
a_\mu^+ \,=\,\Lambda_\mu(\sum_{\nu} a_\nu^-g_\nu)\,=\,\sum_\nu S(\lambda)_{\mu\nu}a^-_\nu.\]
\end{proof}

\begin{remk}
Until now we haven't used the fact that the underlying metric actually is Euclidean
with conical singularities. The preceding construction is fairly general and
can be made on any manifold with conical singularities. Actually, it can be done
in an abstract manner for any symmetric operator with (equal) finite deficiency indices
(compare with section 13.4 of \cite{Grubb}).
\end{remk}

Before coming to the main aim of this paper, which is to understand how much geometric information
is contained in the S-matrix, we derive first two basic properties of $S_{\mu\nu}(\lambda).$

\subsection{Derivative of the S-matrix}\label{secders}
In this section a dot will mean differentiation with respect to $\la,$ and we prove the following lemma.

\begin{lem}\label{lemders}
On $\C\setminus [0,\infty),$ we have
\begin{equation}\label{eqders}
\dot{S}_{\mu\nu}\,=\,\langle G_\nu(\cdot\,;\la),G_\mu(\cdot\,;\overline{\la})\rangle.
\end{equation}
\end{lem}

\begin{proof}
We start from the relation
\[
\left( \Delta_F-\la \right )g_\nu(\cdot\,;\la)\,=\,-\Delta^*F_\nu(\cdot)+\la F_\nu(\cdot),
\]
that we differentiate with respect to $\la$. Since $F_\nu$ doesn't depend on $\la$ and
$g_\nu$ is analytic in $H^2$ we obtain
\[
 \left( \Delta_F-\la \right )\dot{g}_\nu(\cdot\,;\la)\,=F_\nu(\cdot)+g_\nu(\cdot\,;\la)\,=\, G_\nu(\cdot\,;\la).
\]
This gives
\begin{eqnarray*}
\dot{S}(\la)_{\mu\nu}&=& \La_\mu\left( \dot{g}_\nu(\cdot\,;\la)\right) \\
& =& \La_\mu\left( (\dF-\la)^{-1}G_\nu(\cdot\,;\la)\right)\\
&=& \langle G_\nu(\cdot\,;\la), G_\mu(\cdot\,;\overline{\la})\rangle,
\end{eqnarray*}
where we have used \refeq{eq:LambdaG} for the last identity.
\end{proof}

\subsection{Relation with the resolvent kernel}
Denote by $R(x,x'\,; \lambda )$ the resolvent kernel of the Friedrichs extension $\dF.$

Fix $x'\in X_0.$ As a function of the first argument, $R(\cdot,x'\,;\lambda)$ is locally
in $H^2$ near each conical point $p.$ Thus according to \refeq{defab}, there exists
a collection $a^+_\nu(x'\,;\lambda)$ such that, in the neighbourhood of $p$ we have the following
asymptotic expansion :
\begin{equation}\label{eq:AsymptResolv}
R(r\exp(i\theta),x';\lambda)= \sum_{\nu\in N_p} C_\nu a^+_\nu(x'\,;\lambda)r^{|\nu|}\exp(i\nu \theta)+ r_0
\end{equation}
with $r_0\in \overline{\Cc^\infty(X_0)}^{H^2}.$

Using \refeq{eq:Green2}, we see that $a^+_\nu(x'\,;\la)\,=\, \Green(R(\cdot\,,\,x'\,;\la),\,F_\nu)$
and thus, the former expansion may be differentiated with respect to $x'$ in any compact set of
$X_0.$

The following proposition makes the relation between $a^+_\nu(x'\,;\la)$ and $G_\nu(x';\lambda)$
more explicit.

\begin{prop}\label{prop:aG}
For any $x'\in X_0,$ we have
\begin{equation}\label{eq:aG}
G_\nu(x'\,;{\lambda})\,=\,a_\nu^+(x'\,;\lambda)
\end{equation}
where  $a_\nu^+(x';\la)$ is the previously described coefficient in the asymptotic
expansion of $R(\cdot,x';\la)$ near $p$.
\end{prop}

In other words, $G_\nu(x';\la)$ is obtained by selecting in
the resolvent kernel $R(x,x'\,;\la)$ some particular term in the asymptotic behavior
$x\rightarrow p.$ Using $\overline{R(x',x\,;\overline{\la})}\,=\,R(x,x'\,;\la)$
there are similar statements when we fix $x$ and let $x'$ tends to $p.$

\begin{proof}
Denote by $\Delta_1$ the Euclidean Laplace operator on $\Cc^{\infty}(X\setminus (P\cup \{ x'\})).$
This operator fits in the general theory described in section \ref{sec:defFried} by considering that $x'$ actually is the
vertex of a cone of angle $2\pi.$ In particular, Green's formula \refeq{defab} is still valid provided
we take into account $\log$ singularities at $x'.$
The resolvent kernel $R(\cdot,x';\lambda)$ and $G_\nu(\cdot; \la)$ both belong to $\dom(\Delta_1^*).$
the singularities of $R$ are described by the functions $a^+_\nu$ near the conical points and $R$ has a
$\log$ singularity near $x'$ whereas $G_{\nu}$ is smooth near $x'$ and its singular behavior
near the conical points $G_\nu$ is prescribed by \refeq{eq:AsymptResolv}.
Green's formula thus yields :

\[
\langle (\Delta_1^*-\lambda)R(\cdot,x';\la),G_\nu(\cdot\,\overline{\la})\rangle -
\langle R(\cdot,x'\,;\la),(\Delta_1^*-\overline{\la})G_\nu(x'\,;\overline{\la})\rangle\,=\,
\overline{G_\nu(x'\,;\overline{\lambda})}-a_\nu^+(x'\,;\la).
\]

Since $(\Delta_1^*-\lambda)R(\cdot,x'\,;\la)=0=(\Delta_1^*-\overline{\la})G_\nu(x'\,;\overline{\la}),$
we obtain
\[
\overline{G_\nu(x'\,;\overline{\lambda})}\,=\,a_\nu^+(x'\,;\la).
\]
We now use the fact that $G_\nu(x;\lambda)$ is analytic for $\lambda\in \C\backslash [0,\infty)$ and real for
real (and negative) $\lambda.$ Thus by analytic continuation
\[
\overline{G_\nu(x'\,;\overline{\lambda})}\,=\,G_\nu(x',\lambda).
\]
\end{proof}

\section{The S-matrix of E.s.c.s.}
In this section we try to understand what kind of geometric information is encoded
in the S-matrix of a Euclidean surface with conical singularities.
We begin by studying the asymptotic behavior of $S(\la)$ as $\la$ goes to $-\infty.$

\subsection{$S(-|\lambda|)$ for large $\la$.}
It is a general fact that the behavior of the resolvent kernel when $\la$ goes to $-\infty$ is
a local quantity.

This is confirmed by the following lemma.

\begin{lem}
When $\lambda$ goes to $\infty$ then
\[
[S(-|\lambda|)]_{\mu\nu}=O(|\lambda|^{-\infty}),
\]
if $\mu$ and $\nu$ do not correspond to the same conical point.\\
When $\mu$ and $\nu$ correspond to the same conical point $p$ of angle $\alpha$ then we have
\[
[S(-|\la|)]_{\mu\nu}\,=\,[S_\alpha(-|\la|)]_{\mu\nu}+O(|\la|^{-\infty}),
\]
where $S_\alp$ denotes the S-matrix on the infinite cone of total angle $\alpha.$

Moreover both identities may be differentiated with respect to $\la.$
\end{lem}
\begin{proof}
We use the representation of the resolvent kernel using the heat kernel (that we denote here
by ${\mathcal{P}}(t,x,x')$)~:
\begin{equation}\label{heattoresolvent}
R(x,x';-|\la|)\,=\, \int_0^\infty \exp(-t|\la|){\mathcal{P}}(t,x,x') \,dt.
\end{equation}

We now use a standard construction of a parametrix for the heat kernel (see \cite{Cheeger} for instance).
We first enumerate the set of conical points  writing $P:= \{ p_i, ~ 1\leq p_i\leq M\}.$ Then, for each $p_i$
we choose $\tilde{\chi}_i $ and $\chi_i$ two smooth cut-off functions such that $\supp(\chi_i) \subset \{ \tilde{\chi}_i=1\},$
$\chi_i$ is identically $1$ near $p$ and $X$ is isometric to a neighbourhood of the tip of the cone of angle $\theta_{p_i}$
on the support of $\tilde{\chi}_i$. We complete the collections $(\chi_i)_{i\leq M}$ and $(\tilde{\chi}_i)_{i\leq M}$
to $(\chi_i)_{i\leq \tilde{M}},~ (\tilde{\chi}_i)_{i\leq \tilde{M}}$ in such a way that $(\chi_i)_{i\leq \tilde{M}}$ is a
partition of unity, $\tilde{\chi}_i$ is identically $1$ on the support of $\chi_i$ and, for $M < i \leq \tilde{M},\ X$ is
isometric to a neighbourhood of the origin in $\R^2$ on the support of $\tilde{\chi}_i.$ We also set $\mathcal{P}_{i}$ to be the
heat kernel on the cone corresponding to $p_i$ if $i\leq M$ and on the plane otherwise and define
\[
\tilde{{\mathcal{P}}}(t,x,x')\,=\, \sum_{i=1}^{\tilde{M}} \tilde{\chi}_i(x){\mathcal{P}}_i(t,x,x')\chi_i(x).
\]

Using Duhamel's principle and the fact that $\mathcal{P}_i$ fastly decays away of the diagonal
(see eq (1.1) of \cite{Cheeger}) yields that
$\tilde{{\mathcal{P}}}(t)-\mathcal{P}(t)$ maps $L^2$ into $H^s$ for any $s,$ and
\[
\| \tilde{{\mathcal{P}}}(t)-\mathcal{P}(t)\|_{L^2\rightarrow H^s} \,=\, O(t^\infty)
\]
when $t$ goes to $0,$ so that $\tilde{\mathcal{P}}$ is a parametrix for the heat kernel.

Inserting into \refeq{heattoresolvent} and integrating against $f_\nu$ we obtain
\[
g_\nu(x;-|\la|)\,=\,\tilde{\chi}_i(x)\int_0^\infty \int_X \mathcal{P}_i(t,x,x')f_{\nu}(x';-|\la|)dS(x')dt + r_\la(x),
\]
where the remainder $r_\la\in H^2$ and $\|r_\la\|_{H^2}\,=\,O(|\la|^{-\infty})$
and the index $i$ corresponds to the conical point corresponding to $\nu.$
The first statement follows. The second also follows by remarking that $F_\nu, f_\nu$ and $\Lambda_\nu$
can also be seen as living on the cone and that the latter equation is also valid on the complete cone.
Differentiating with respect to $\la$ amounts to replace $\mathcal{P}$ by $\Delta_F\mathcal{P}$
and we can use the same argument.
\end{proof}

Using example \ref{SonCone} we obtain the following proposition as a corollary.
\begin{prop}\label{prop:Sasymp}
When $\lambda$ goes to $\infty$ we have
\[
\begin{array}{lcl}
[S(-|\lambda|)]_{\mu\nu} &= &O(|\lambda|^{-\infty})~~~~~~~~\mbox{if}~~~~~~ \mu \neq \nu,\\%
{[S(-|\lambda|)]}_{\nu\nu}
&= & -\frac{\Gamma(1-|\nu|)}{2^{2|\nu|}\Gamma(1+|\nu|)}\cdot |\la|^{|\nu|}\,
+\,O(|\lambda|^{-\infty}),~~~~~\mbox{if}~~~~~~ \nu\neq 0,\\
{[S(-|\lambda|)]}_{00} &=&  \frac{1}{2}\ln(|\la|) \,-\,\left( \ln(2)-\gamma\right)\,+\,O(|\lambda|^{-\infty}).\\
\end{array}
\]
\end{prop}

\begin{remk}\label{rk:difforbit}
It would be interesting to study the asymptotic behavior of $S(\pm i |\lambda|).$ It is then expected
to see contributions of periodic diffractive orbits (compare with \cite{Hillairet}).
\end{remk}

\subsection{Explicit formulas for $S(0)$}\label{sec:exfor}
In this subsection we will show that for $\nu\neq 0$ the coefficient
$S_{\mu\nu}(\la)$ is continuous at $\la=0$ and may be expressed using
standard objects of the Riemannian surface $X.$

Recall that, in the distinguished local parameter $\zeta$ near
some conical point $P$ the asymptotic expansion was given in \refeq{eq:defabdlp}.
It follows that we have

\begin{equation*}
\left \{
\begin{array}{lr}

F_\nu(\zeta,\overline{\zeta}) \sim C_\nu \zeta^{-k} & k>0,\\
F_\nu(\zeta,\overline{\zeta}) \sim C_\nu \overline{\zeta}^k & k<0,
\end{array}
\right .
\end{equation*}
where, as usual $\nu$ and $k$ are related by the relation $\nu\, = \, \frac{2\pi}{\theta_p}\cdot k.$

We first prove the following lemma.

\begin{lem}
If $\nu\neq 0$ then $G_\nu(\cdot; \la)$ is continuous at $\la=0$
and $G_\nu(\cdot;0)$ is a harmonic $L^2$ function on
$X$ such that
\begin{equation*}
\left \{
\begin{array}{lr}
G_\nu(\zeta,\overline{\zeta} ; 0)\,=\,\zeta^{-k}\,+\,O(1) & k>0,\\
G_\nu(\zeta,\overline{\zeta};0)\,=\, \overline{\zeta}^k\,+\,O(1) & k<0.
\end{array}
\right .
\end{equation*}
\end{lem}

\begin{proof}
Recall that we have set $G_\nu(\cdot ;\la)\,=\, F_\nu+g_\nu(\cdot ; \la),$
where $g_\nu(\cdot;\la)$ is the unique solution to
\[
(\Delta_F-\la)\,g_\nu(\cdot;\la)\,=\,-\left( \Delta^* -\la\right)F_\nu.
\]
Since $\dis\int_X \left( \Delta^* -\la\right)F_\nu \,dx\,=\,0$ the continuity
at $0$ follows from the fact that the $\ker(\Delta_F)$ consists only in the
constant function. By continuity we obtain $G_\nu(\cdot\,;0)$ is a solution to
$\Delta^* G_\nu(\cdot\,;0)=0$ and, therefore, $G_\nu(\cdot\,;0)$ is
harmonic on $X_0.$
\end{proof}

\begin{remk}
Let $\zeta$ be denoting the distinguished local parameter near a fixed $p\in P$. The problems
\begin{equation}\label{dir1}
\begin{cases}\Delta U_k=0\ \ \ {\text on}\ \ \ X\setminus P\\
U_k\sim \zeta^{-k}+O(1),\  \text{as} \ \ \zeta\to 0
\end{cases}
\end{equation}

for $0 < k < \frac{\theta_p}{2\pi}$ and

\begin{equation}\label{dir2}
\begin{cases}\Delta U_k=0\ \ \ {\text on}\ \ \ X\setminus P\\
U_k \sim \overline \zeta^k+O(1),\  \text{as} \ \ \zeta\to 0\\
\end{cases}
\end{equation}
for $-\frac{\theta_p}{2\pi} < k <0$ have solutions only up
to an additive constant.
On the other hand, the problem
$$
\begin{cases}\Delta u=0\ \ \ {\text on}\ \ \ X\setminus P\\
u\sim \log r+O(1),\  \text{as} \ \ \zeta\to 0
\end{cases}
$$
doesn't have a solution. Thus the behaviour of the coefficients $S_{0\nu}(\lambda)$ and $S_{\mu0}(\lambda)$
may not even be properly defined for $\lambda=0.$ When writing  $S(0)$ we will implicitly
assume that only the coefficients $S_{\mu\nu}$ with nonzero $\mu$ and $\nu$ are considered (see also remark \ref{rk:DregL}).
\end{remk}

In the next subsection we construct solutions to the problems
(\ref{dir1}, \ref{dir2}) since they give the functions
$G_\nu( \cdot\,;\,0)$ from which the coefficients $S_{\mu\nu}$
can be computed (for nonzero $\mu$ and $\nu$).

\subsection{Special solutions and an explicit expression for $S(0)$}\label{sec:S0}
Choose a canonical basis of cycles, $\{a_\alpha, b_\alpha\}$ on
the Riemann surface $X$ and let $\{v_\alpha\}_{\alpha=1, \dots, g}$ be the
corresponding basis  of holomorphic normalized differentials. Let
${\mathbb B}$ be the matrix of $b$-periods of $X$.

We have the following proposition.

\begin{prop}\label{specresh}
Fixing $P$ a conical point and $k\in {\mathbb N},$ there exist $\Omega_k$
and $\Sigma_k$ such that
\begin{enumerate}
\item $\Omega_k$ and $\Sigma_k$ are meromorphic differentials
of the second kind on $X$ with only one pole of order $k+1$ at $P.$
\item In the distinguished local parameter near $P$, they satisfy
\begin{equation}
\label{genSch}
\left\{
\begin{array}{c}
\Omega_k(\zeta)=-\frac{k}{\zeta^{k+1}}d\zeta+O(1) \\
\\
\Sigma_k(\zeta)=-\frac{ik}{\zeta^{k+1}}d\zeta+O(1).
\end{array}
\right .
\end{equation}
\item All the $a$ and $b$-periods of $\Omega_k(P,\cdot)$ and $\Sigma_k(P, \cdot)$ are purely
imaginary.
\end{enumerate}
\end{prop}

\begin{proof}
Let $\omega(\cdot, \cdot)$ be the canonical meromorphic bidifferential
on the Riemann surface $X$ (see \cite{Fay}, p. 3), for which the following
asymptotic expansion holds
$$\frac{\omega(\zeta(Q_1),
\zeta(Q_2))}{d\zeta(Q_1)d\zeta(Q_2)}=\frac{1}{(\zeta(Q_1)-\zeta(Q_2))^2}+\frac{1}{6}S_B(\zeta(Q_2))+o(1)$$
as $Q_1\to Q_2$, where $S_B$ is the Bergman projective connection.
Moreover, $\omega$ is normalized in such a way that
\begin{equation}\label{normomega}
\left\{
\begin{array}{l}
\displaystyle \oint_{a_\alpha}\frac{\omega(\cdot,
\zeta)}{d\zeta}\Big|_{\zeta=0}=0 \\
\\
\displaystyle \oint_{b_\alpha}\frac{\omega(\cdot,
\zeta)}{d\zeta}\Big|_{\zeta=0}=2\pi i
\frac{v_\alpha(\zeta)}{d\zeta}\Big|_{\zeta=0}\,,
\end{array}
\right .
\end{equation}
for $\alpha=1, \cdots, g$.

Let $(c_\alp)_{\alp=1\cdots g}$ be coefficients to be chosen later
and consider the meromorphic differential
\begin{equation}\label{differ}
-\frac{\omega(\cdot,
\zeta)}{d\zeta}\Big|_{\zeta=0}+\sum_{\alpha=1}^gc_\alpha
v_\alpha\,.\end{equation}

We want to choose $c_\alp$ in (\ref{differ}) so that all the $a$- and $b$-periods of this differential
are purely imaginary. The vanishing of the real parts of all $a$-periods
implies that all the constants $c_\alpha$ are purely imaginary.
The vanishing of the real part of the period over the cycle $b_\beta$ then gives :
\[
\text {Re}\,\left( \oint_{b_\beta} \sum c_\alp v_\alp \right)\,=\,
\text{Re}\,\left( \oint_{b_\beta}\frac{\omega(\cdot,
\zeta)}{d\zeta}\Big|_{\zeta=0} \right).
\]
Using the fact that the $c_\alp$ are known to be purely imaginary and the normalization
of $\omega$ recalled in \refeq{normomega} we obtain the following system of equations :
\begin{equation}\label{defcalpha}
\sum_{\alp=1}^g [\text{Im}\, \B]_{\beta \alp}c_\alp \,=\,2\pi i
\text{Im}\,\left( \frac{v_\beta}{d\zeta}\Big|_{\zeta=0}\right)
\end{equation}
Since $\text{Im}\,(\B)$ is invertible, this uniquely determines $c_\alp.$

In order to get $\Sigma_1$ we apply the same method searching coefficients
$c_\alp$ so that the meromorphic differential
\[
-i\frac{\omega(\cdot,\zeta)}{d\zeta}\,+\,\sum_{\alpha=1}^g c_\alpha v_\alpha
\]
has purely imaginary periods. The system of equations we obtain
is similar to \refeq{defcalpha} except that $\text{Im}\,\left( \frac{v_\beta}{d\zeta}\Big|_{\zeta=0}\right)$
is replaced by $\text{Re}\,\left( \frac{v_\beta}{d\zeta}\Big|_{\zeta=0}\right).$
It still has a solution using the same invertibility of $\text{Im}\,(\B).$

To get $\Omega_k$ and $\Sigma_k$ with an arbitrary
$k\geq 1$ we repeat the same construction taking the first
term in (\ref{differ}) to be
\[
\frac{(-1)^k}{(k-1)!}\left[\frac{d}{d\zeta}\right]^{k-1}\frac{\omega(\cdot,
\zeta)}{d\zeta}\Big|_{\zeta=0}\,.
\]
We will obtain an equation similar to \refeq{defcalpha} so that
eventually, the existence result thus follows from the existence of $\omega$ and the
fact that the matrix $\text{Im}\,({\mathbb B})$ is invertible.
\end{proof}

This proposition gives the following corollary.

\begin{coro}
Let $\Omega_k$ and $\Sigma_k$ be defined by the preceding proposition, then the following
formula defines a function $f_k$ which is harmonic in $X\setminus\{ P\} :$
\begin{equation}
f_k(Q)\,=\, \text{Re}\,\left\{ \int_{P_0}^Q \Omega_k\right\}-i\text{Re}\,\left\{ \int_{P_0}^Q \Sigma_k\right\}.
\end{equation}
Moreover, in the distinguished local parameter near $P,$ $f_k$ admits the following
asymptotic behavior :
\[
f_k(\zeta)\,=\,\frac{1}{\zeta^k}+O(1).\]
\end{coro}

\begin{proof}
Since all the $a-$ and $b-$ periods of $\Omega$ and $\Sigma$ are purely imaginary,
$f_k$ is indeed well-defined on $X.$ The remaining statements follow from the construction.
\end{proof}

By considering $C_\nu f_k$ or $C_\nu \overline{f_k}$ we obtain the functions $G_\nu(\cdot;0)$ up to an additive constant.
This additive constant is harmless when computing the matrix elements $S_{\mu\nu}(0).$

\subsubsection{Examples}
\begin{enumerate}
\item {\bf A conical point of angle $2\pi<\beta\leq 4\pi$ on a Euclidean surface of genus $\geq 1$.}
In this case one has $n=1.$

The proposition \ref{specresh} combined with the asymptotics of $\omega$ yield
\begin{equation}\label{A}
\begin{split}
 \int_{P_0}^\zeta\Omega_1(P,
\cdot) = \frac{1}{\zeta}+c_0 + & \left[-\frac{1}{6}S_{B}(\zeta)\Big|_{\zeta=0}+2\pi
i\sum_{\alpha=1, \beta=1}^g((\text{Im}\, {\mathbb
B})^{-1})_{\alpha\beta}\text{Im}\,\left\{\frac{v_\beta(\zeta)}{d\zeta}\Big|_{\zeta=0}\right\}\frac{v_\alpha(\zeta)}{d\zeta}
\Big|_{\zeta=0}\right]\zeta\\
& +O(\zeta^2)
\end{split}
\end{equation}
with some constant $c_0,$ and
\begin{equation}\label{B}
\begin{split}
\int_{P_0}^\zeta\Sigma_1(P,
\cdot)=\frac{i}{\zeta}+ d_0+&\left[-\frac{i}{6}S_{B}(\zeta)\Big|_{\zeta=0}+2\pi
i\sum_{\alpha=1, \beta=1}^g((\text{Im}\, {\mathbb
B})^{-1})_{\alpha\beta}\text{Re}\,\left\{\frac{v_\beta(\zeta)}{d\zeta}\Big|_{\zeta=0}\right\}\frac{v_\alpha(\zeta)}{d\zeta}\Big|_{\zeta=0}\right]\zeta\\
&+O(\zeta^2)
\end{split}
\end{equation}
with some constant $d_0$.

Denoting the expressions in square brackets in (\ref{A}) and
(\ref{B}) by $A$ and $B$ respectively, one gets the asymptotics

$$f_1(\zeta, \overline \zeta)=\frac{1}{\zeta}+{\rm
const}+\frac{A-iB}{2}\zeta+\frac{\overline A-i\overline
B}{2}\overline\zeta+O(|\zeta|^2)$$ and, therefore,

\begin{equation}\label{scattering}
S_p(0)=\left(\begin{matrix} * \ \ \ \ \    * \ \ \ \ \     *\\
* \ \  \frac{A-iB}{2}\     \frac{\overline A-i\overline B}{2}\\
*  \ \ \frac{A+iB}{2} \    \frac{\overline A+i \overline B}{2} \end{matrix} \right) ,\end{equation}

where the index $p$ means that we have written down only the coefficients of $S(0)$ that
corresponds to indices $\nu\in N_p$

\item {\bf A Euclidean sphere with one $4\pi$ singularity and six $\pi$ singularities.}
Consider the surface of example \ref{eg:spherepant} i.e. the Riemann sphere
with  metric
$$\frac{|z|^2|dz|^2}{\prod_{k=1}^6|z-z_k|}.$$

We consider the part of the $S$-matrix with non-zero indices $\mu$ and $\nu.$
We thus only have to consider the asymptotic behavior near $0$ and compute the
coefficients $S_{\und \und},~S_{-\und-\und},~S_{-\und \und}$ and $S_{\und -\und}.$

The distinguished local parameter $\zeta$  in a vicinity of the conical point $z=0$
is given by
$$\zeta(z)=\left(\int_0^z\frac{w\,dw}{\sqrt{\prod_{k=1}^6(w-z_k)}}\right)^{1/2}\,.$$

The special solution
$f_1$ is now not only  harmonic but even holomorphic in ${\mathbb C}P^1\setminus 0$
and is nothing but the function $A/z$ with some constant $A$.

One has
$$\frac{A}{z}=\frac{1}{\zeta}+{\rm const} +S_{\und \und}(0)\zeta+O(\zeta^2),$$
Therefore, $A=\frac{dz}{d\zeta}\Big|_{\zeta=0}$ and a simple calculation shows that
$$S_{\und \und}(0)=-\frac{1}{6}\frac{z'''(\zeta)z'(\zeta)-\frac{3}{2}(z''(\zeta))^2}{(z'(\zeta)^2}\Big|_{\zeta=0}=
-\frac{1}{6}\{z, \zeta\}|_{\zeta=0}\,,$$
where $\{z, \zeta\}$ is the Schwarzian derivative.
One has also $S_{-\und-\und}(0)=\overline{S_{\und \und}(0)}$ and
$S_{\und-\und}(0)=S_{-\und \und}(0)=0$.

In the very symmetric case where the $z_k$ form a regular hexagon, the computation yields that
$z = c\cdot \zeta(1+O(\zeta^6))$ so that $S_{\und \und}$ and $S_{-\und -\und}$ also vanish.
\end{enumerate}

\subsection{$S$-matrix as a function on the moduli space of
holomorphic differentials: variational formulas}
Let $(X, \omega)\in H_g(k_1, \dots, k_M)$ and let $S(\lambda)$ be the $S$-matrix
corresponding to a conical point of the translation surface $(X, |\omega|^2)$
(i. e. one of the zeros of the holomorphic one-form $\omega$). Here we derive
the variational formulas for $S(\lambda)$ with respect to Kontsevich-Zorich
homological coordinates on $H_g(k_1, \dots, k_M)$.
\begin{prop} Let $z(p)=\int^p\omega$. Introduce the following (closed) (1-1)-form on $X_0$:
$$\Theta_{\mu \nu}=[{G_\mu( z; {\lambda})}]_{z\overline z}G_\nu(z ; \lambda)d\overline z+
[{G_\mu(z ; {\lambda})}]_z[G_\nu( z ; \lambda)]_zdz\,,$$
Then the variational formulas hold:
\begin{equation} \label{e1}\frac{\partial S_{\mu \nu}(\lambda)}{\partial A_\alpha}
=2i\oint_{b_\alpha}\Theta_{\mu \nu}; \ \ \ \ \alpha=1, \dots, g,\end{equation}
\begin{equation}\frac{\partial S_{\mu \nu}(\lambda)}{\partial B_\alpha}%
=-2i\oint_{a_\alpha}\Theta_{\mu \nu}; \ \ \ \ \alpha=1, \dots, g,\end{equation}
\begin{equation}\label{e3}\frac{\partial S_{\mu \nu}(\lambda)}{\partial z_k}%
=2i\oint_{p_k}\Theta_{\mu \nu}; \ \ \ \ k=2, \dots, M,\end{equation}
where the integrals in (\ref{e3}) are taken over some small contours encircling conical points $p_k$.
\end{prop}

{\bf Proof.}
The method of proof follows closely \cite{KokKor}. We will prove only the variational formulas
with respect to coordinates $A_\alpha$ since the other formulas can be established similarly.

According to \cite{KokKor} (Proposition 2, p. 84) one has
\begin{equation}\label{varres}
\partial_{A_\alpha}R(x, y; \lambda)=2i\oint_{b_\alpha}R(x, z; \lambda)
R_{z\overline z}(y, z; \lambda)d\overline z+
R_z(x, z; \lambda)R_z(y, z; \lambda)\,dz\,.
\end{equation}
(Here $R(x, y; \lambda)$ stands for the resolvent kernel of the Friedrichs extension; one has
$R_{z\overline z}(x, z; \lambda)=\frac{\lambda}{4}R(x, z; \lambda)$. )
On the other hand, by definition of $g_\nu$ we have
\begin{equation}\label{luc11}
g_\nu(x\,;\la)=\,-\iint_X[R(x,y\,;\la)(\Delta-\lambda)F_\nu(y)]dy;
\end{equation}
and
Lemma 7 on page 88 of \cite{KokKor} reads as
\begin{equation}\label{eq:varKK}
\partial_{A_\alpha}\iint_X \Phi(x, \overline x; \text{moduli})dx
=\iint_X \partial_{A_\alpha}\Phi(x, \overline x, \text{moduli})dx+
\frac{i}{2}\oint_{b_\alpha}\Phi(x, \overline x, \text{moduli})d\overline x\,.
\end{equation}

The cycle $b_\alpha$ does not intersect the support of $F_\nu$ and the terms $F_\nu$ and
$(\Delta-\lambda)F_\nu$ are moduli independent, therefore,
$$\partial_{A_\alpha}G_\nu(x\,;\la)=\partial_{A_\alpha}(F_\nu+g_\nu)=
\partial_{A_\alpha} g_\nu(x\,;\la).$$
Using \refeq{luc11} and \refeq{eq:varKK}, we obtain
\begin{equation*}
\begin{split}
\partial_{A_\alpha}G_\nu(x\,;\la) & =%
2i\iint_X dy [(\Delta-\lambda)F_\nu(y)]\oint_{b_\alpha}%
\left\{R(x, z; \lambda)R_{z\overline{z}}(y, z; \lambda)d\overline z+R_z(x, z; \lambda)
R_z(y, z; \lambda)dz\right\}\\
& = 2i\oint_{b_\alpha}R(x, z; \lambda)\left[\iint_X \frac{\lambda}{4}R(y, z; \lambda)
(\Delta-\lambda)F_\nu(y)dy\right]d\overline{z}\\
& \quad +\,2i\oint_{b_{\alpha}}R_z(x, z; \lambda)\left[\iint_X R_z(y, z; \lambda)
(\Delta-\lambda)F_\nu(y)dy\right]dz\\
& = 2i\oint_{b_\alpha}R_{z\overline z}(x, z; \lambda)g_\nu(z ; \lambda)d\overline z + %
R_z(x, z; \lambda)[g_\nu(z ; \lambda)]_zdz\\
& = 2i\oint_{b_\alpha}R_{z\overline z}(x, z; \lambda)G_\nu(z ; \lambda)d\overline z +%
R_z(x, z; \lambda)[G_\nu(z ; \lambda)]_zdz
\end{split}
\end{equation*}

We finally obtain
$$\partial_{A_\alpha} g_\nu(\lambda, x)=2i\oint_{b_\alpha}
R_{z\overline z}(x, z; \lambda)G_\nu(z; \lambda)d\overline z
+R_z(x, z; \lambda)[G_\nu(z ; \lambda)]_zdz.$$
Using proposition \ref{prop:aG} and equation \refeq{eq:AsymptResolv}
to identify the behavior near the conical points of the
different terms we obtain
\[
\partial_{A_\alpha}S_{\mu\nu}\,=\,2i\oint_{b_\alpha}
\left[a^+_\mu(z;\la)\right]_{z\overline{z}}G_{\nu}(z;\la)d\overline{z}%
+\left[a^+_\mu(z;\la)\right]_{z}\left[G_{\nu}(z;\la)\right]_zd{z}.
\]
Using proposition \ref{prop:aG}, this gives the result.

\section{Krein's formula and relative determinants}\label{sec:Kfrd}
There are several ways of defining determinants of operators
acting on an infinite dimensional space. We recall the following two
basic constructions : first a {\em perturbative determinant} when the operator is a
trace-class perturbation of the identity, and
second {\em zeta-regularization} which is used in particular for
Laplacian-like operators.

Both these approaches can also be used to define relative determinants when comparing
two operators $H_0$ and $H_1$ in which one is thought to be a perturbation
of the other. Krein's formula is a classical tool in this setting and usually applies
when the difference $f(H_1)-f(H_0)$ is trace-class for some simple function $f$.
In that case it is possible to define a relative perturbative determinant
(see \cite{Yafaev}). This approach applies well to the case when $H_0$ and $H_1$
are different self-adjoint extensions of a symmetric operator that has finite
deficiency indices. Indeed, in that case the difference of the resolvents is a
finite-rank operator and, moreover, the perturbative determinant is
actually the determinant of a finite dimensional matrix.

We will thus adapt these techniques to our setting. The method is clearly
identified in the literature (see \cite{Yafaev} and also \cite{BFK}) and
the main task here consists in identifying the perturbative determinant
in terms of the boundary condition and the $S$-matrix.

Once this is done, we will use this determinant
to define a zeta-regularization and compare
the determinants that are obtained this way.

\begin{remk}
We insist here that we will actually use the perturbative determinant
to show that zeta-regularization is possible and then to compare the
two definitions of determinants. In particular, all the issues
that are relative to zeta-regularization may be expressed using
the perturbative determinant (when the latter can be defined).
\end{remk}

\subsection{Krein's formula and perturbative determinant}
One convenient way of parametrizing the self-adjoint extensions
of $\Delta$ is by using two matrices $P$ and $Q$ in the following way
(see \cite{KosSch}).

We first construct two vectors $A^{\pm}$ that collect the coefficients
$a^{\pm}_\nu.$ We organize these coefficients so that the first $n_{p_1}$
entries correspond to the first conical point $p_1$ then we put the data
corresponding to the second conical points and so on.

A lagrangian subspace $L$ in $\dom(\Delta^*)/\dom(\Delta)$
can be parametrized by a system of linear equations of the following form :
\[
PA^-+QA^+ =0,
\]
where $P$ and $Q$ are square matrices satisfying
$\mbox{rank}(P,Q)$ is maximal and $P^*Q$ is self-adjoint.
We fix two such matrices and denote by $\dL$ the corresponding
self-adjoint extensions.

 It is
possible to find a basis in which the $n\times 2n$ matrix $(P\,Q)$ has the
following block-decomposition (\cite{KosSch}) :
\begin{equation}\label{eq:Bkdec}
\left(
\begin{array}{cccc}
P_2 & P_3 & Q_1 & 0 \\
0   & P_1 & 0 & 0
\end{array}
\right),
\end{equation}
in which $P_1$ and $Q_1$ are invertible and $L:=Q_1^{-1}P_2$ is self-adjoint.

\begin{defn}\label{def:reg}
We will call an extension $\dL$ {\em regular} if functions
in $\dom(\dL)$ are not allowed to have logarithmic singularities.
Equivalently, $\dL$ is regular if and only if for any $u\in\dom(\dL)$,
and any conical point $p,$ the coefficient $a_{p,0}^-$ of $u$ vanishes.
\end{defn}

The following observation (based on the classical Krein formula) is the key technical result of the present paper.

\begin{prop}\label{prop:difftrace}
For any $\la\in \C\setminus (\spec(\dF)\cup\spec(\dL))$ the following identity
holds :
\[
\Tr\left( (\dL-\la)^{-1}-(\dF-\la)^{-1}\right )\,=\,-\Tr\left( (P+QS(\la))^{-1}Q\dot{S}(\la) \right),
\]
where the dot indicates derivation with respect to $\la.$
\end{prop}

\begin{proof}
Let $\lambda$ be in the union of the resolvent sets of $\dF$ and $\dL,$ and let $f$ be
in $L^2.$ We search a matrix $X=[x_{\mu\nu}]$ such that
we have the following Krein formula (see, e. g.,  \cite{BirSol} or \cite{Albeverio}, Theorem A.3)
\begin{equation}\label{Kreinpretrace}
(\dL-\la)^{-1}f\,=\,(\dF -\la)^{-1}(f) +
\sum_{\mu,\nu}x_{\mu\nu}G_\mu(\cdot\,;\la)\Lambda_\nu\left[(\dF -\la)^{-1}(f)\right ].
\end{equation}

We denote by $u=(\dF-\la)^{-1}(f)$ and we compute the vectors $A^\pm$ of the right-hand side
\begin{gather*}
a^-_\mu\,=\, \sum_{\nu}x_{\mu\nu} \La_\nu(u),\\
a^+_{\mu'} \,=\,\La_{\mu'}(u)+\sum_{\mu,\nu} x_{\mu\nu}[S(\la)]_{\mu'\mu}\Lambda_\nu(u).
\end{gather*}

Denoting by $\La$ the vector $\La_\nu(u)$ we thus have
\[
A^-\,=\,X\La,~~~A^+=(I + S(\lambda)X)\La.
\]
Plugging into the self-adjoint condition we obtain that the following relation is satisfied.

\[
\left[PX +Q(I+S(\la)X)\right]\cdot \La\,=\,0.
\]
Using the block decomposition \refeq{eq:Bkdec}, we see that
\[
P+QS(\la)\,=\,\left(
\begin{array}{cc}
P_2+Q_1S(\la) & *\\
0& P_1
\end{array}
\right)
\]
Since $\lambda$ is in both resolvent sets, $\La$ is arbitrary and the
preceding system always has a solution. We obtain that $(P_2+Q_1S({\la}))$ must be invertible
and hence $P+QS(\la).$
Finally, we obtain
\[
X\,=\,-(P+QS({\la}))^{-1}Q.
\]

Denoting by $\Pi_{\mu\nu}(\la)$ the (rank one) operator defined from $H^2$ into $L^2$ by
\[
\Pi_{\mu\nu}(\la)(u)\,=\, G_\mu(\cdot\,;\la)\Lambda_\nu(u),
\]
equation \refeq{Kreinpretrace} may be rewritten :
\[
(\dL-\la)^{-1}-(\dF-\la)^{-1}\,=\,\sum_{\mu,\nu} x_{\mu\nu}\Pi_{\mu\nu}(\la)\circ(\dF-\la)^{-1}
\]

Observe that the right-hand side is finite rank so that we can trace this equation and obtain
\[
\Tr\left((\dL-\la)^{-1}-(\dF-\la)^{-1})\right)\,=\,\sum_{\mu, \nu}x_{\mu\nu}
\Tr\left( \Pi_{\mu\nu}(\la)\circ(\dF-\la)^{-1} \right).
\]

Using lemma \ref{lemtracepi} below and lemma \ref{lemders} we obtain
\begin{eqnarray*}
\Tr\left((\dL-\la)^{-1}-(\dF-\la)^{-1})\right) &=&\sum_{\mu, \nu}x_{\mu\nu}
\langle G_\mu(\cdot\,;\la),G_\nu(\cdot\,;\overline{\la})\rangle \\
&=& \sum_{\mu \nu}x_{\mu\nu} [\dot{S}(\la)]_{\nu\mu}\\
&=& -\Tr\left( (P+QS(\la))^{-1}Q\dot{S}(\la) \right).
\end{eqnarray*}
\end{proof}

It remains to prove the following lemma.

\begin{lem}\label{lemtracepi}
The trace of the rank one operator $\Pi_{\mu\nu}(\la)\circ(\dF-\la)^{-1}$ is given by
\[
\Tr \left( \Pi_{\mu\nu}(\la)\circ(\dF-\la)^{-1}\right)\,=\,
\langle G_\mu(\cdot\,;\la),G_\nu(\cdot\,;\overline{\la})\rangle
\]
\end{lem}

\begin{proof}
Let $e_n$ be an orthonormal basis of $L^2$ then
\begin{eqnarray*}
\langle \Pi_{\mu\nu}(\la)\circ(\dF-\la)^{-1} e_n,e_n\rangle &=&
\langle G_\mu(\cdot\,;\la),e_n\rangle \cdot \Lambda_\nu\left((\dF-\la)^{-1}e_n\right) \\
&=& \langle G_\mu(\cdot\,;\la),e_n\rangle\cdot
\langle e_n, G_\nu(\cdot\,;\overline{\la})\rangle
\end{eqnarray*}
Summing over $n$ and using Parseval's identity gives the lemma.
\end{proof}

We may now define $D$ on the union of the resolvent sets of $\dF$ and $\dL$ by
\begin{equation}
  \label{eq:defD}
  D(\la)\,=\,\det \left(P+QS(\la)\right).
\end{equation}

\begin{remk}\label{rk:DregL}
When the extension is regular the matrix $P+QS(\la)$ doesn't involve
the coefficients $S_{\mu\nu}$ whenever $\mu$ or $\nu$ is $0$ (because
these are multiplied by a zero entry of $Q$). Hence the matrix $P+QS(0)$ makes
perfect sense and can be computed using the results of section \ref{sec:exfor}.
\end{remk}

The preceding proposition gives
\begin{equation}
  \label{eq:difftraceD}
 \Tr\left( (\dL-\la)^{-1}-(\dF-\la)^{-1}\right )\,=\,-\frac{D'(\la)}{D(\la)}
\end{equation}
This implies that $\frac{D'}{D}$ extends to a meromorphic function with poles
that correspond to eigenvalues of $\dL$ and $\dF$ and with residues
$\dim(\ker(\dL-\la))-\dim(\ker(\dF-\la)).$

Since $\frac{D'}{D}$ is the logarithmic derivative of $D$, it is convenient to give a
name to $\ln(D).$
We thus denote by $\Omega\subset \C$ the set obtained by removing a downward vertical
cut starting at each eigenvalue of $\dF$ and $\dL$ i.e.
\[
\Omega \,=\,\C\setminus\left\{ \la-it,\,\la\in {\spec}(\dF)\cup{\spec}(\dL),\,t\in(-\infty,0]\,\right\},
\]
and, on $\Omega$, we define the function $\xit$ by
$\xit(\la):=\,-\frac{1}{2i\pi} \ln\left(\det(P+QS(\la))\right).$

Observe that on $\Omega$ we have, by definition,
\begin{equation}\label{eq:BKdef}
D(\la)\,=\,\exp\left( -2i\pi \xit(\la)\right).
\end{equation}

The function $\xit$ is intimately related to the spectral shift function
$\xi$ (see \cite{Yafaev, GesSim}). Although the latter is usually used in settings
with continuous spectrum, it is possible to define it even when $H_0$ and $H_1$ have
pure point spectrum. In the latter case, it follows from the definitions that $\xi$ is the step-function :
$\xi(t):= N_1(t)-N_0(t)$ where $N_i$ is the counting function associated with $H_i.$

It follows from our definition of $\xit$ that the function $\xi$ defined on $\R$ by
\[
\xi(t)\,:=\, -\frac{1}{2\pi}\mbox{Arg}(D(t))\,=\,~\mbox{Re}\, \xit(t)
\]

is a step function with jumps located at the eigenvalues of $\dF$ and $\dL.$
Moreover the jumps are exactly the differences $\dim(\ker(\dL-\la))-\dim(\ker(\dF-\la)).$
We thus obtain the spectral shift function of $\dF$ and $\dL.$ (compare with \cite{Yafaev}
Thm 1 p. 272).

In our setting Birman-Krein formula would be
\refeq{eq:BKdef} (compare with \cite{Yafaev} p.272) and
would follow, in our case, from our definitions.
In the next subsection we will
prove that, using $D$, one may define a determinant of $\dL$ via
zeta-regularization and then establish the relation :

\begin{equation}
  \label{eq:BK}
  {\det}_\zeta(\dL-\la)\,=\,C_0\cdot D(\la) {\det}_\zeta(\dF-\la),
\end{equation}

in which $C_0$ is some constant that
we will also determine.

In particular, we will now prove that $D$
allows us to recover the 'exotic' features of the zeta function associated with
$\dL.$ This unusual behavior has been extensively studied
in \cite{KLP} in a setting very close to ours and in \cite{GKM} in greater
generality. Our main contribution here is the interpretation of $D$ using
$S$-matrix that, in some sense, gives a geometrical interpretation to
the 'secular equation' method of \cite{KLP}.

\subsection{Comparing determinants}

The procedure here is not as straightforward as usual because of unusual behavior
of the zeta function near $s=0.$ In particular, $\zeta(s,\dL)$ will admit a analytic continuation
that is regular at $0$ only if $L$ is regular (though with possible unusual poles).
This unusual behavior as we just mentioned has been extensively studied in literature (see \cite{GKM, KLP, LMP});
from our point of view, it is linked with the asymptotic behavior of $D(\la)$
for large negative $\la.$ We thus begin by deriving this asymptotic expansion.

\subsubsection{$D(\la)$ for large negative $\la$.}\label{sec:mularge}

The analysis that follows is closely related to the one performed in \cite{KLP}. This is not surprising
since the asymptotic regime $\la$ goes to $-\infty$ is local. In particular, the
function $D(-|\la|):=\det(P+QS(-|\la|))$ on a cone has to be compared to the
function $F(i\sqrt{|\la|})$ in \cite{KLP}.

We first use prop. \ref{prop:Sasymp} and consider all possible sums of
the exponents $\nu_i$ that appear in this proposition. This gives us a collection
of numbers that we order and denote by $\alp_0\,>\,\alp_1\,>\,\cdots\,>\,\alp_k\,>\,\cdots.$
Expanding now the determinant, and ordering the terms, we get
\[
D(-|\la|)\,=\, \sum_{finite} a_{kl} |\la|^{\alp_k}(\ln |\la|)^l + O(|\la|^{-\infty}).
\]

By definition, there are no logarithm in the expansion corresponding to a regular
self-adjoint extension, therefore, in that case, the expansion reads :
\[
D(-|\la|)\,=\, \sum_{finite} a_{k} |\la|^{\alp_k}\,+\, O(|\la|^{-\infty}).
\]

We set $l_k$ the largest integer $l$ such that $|\la|^{\alp_k}(\ln |\la|)^l$ appears
in the expansion and we set $\bet_k={\alp}_0-\alp_k$ we have

\[
D(-|\la|)\,=\,a_{0l_0} |\la|^{\alp_0}(\ln |\la|)^{l_0}
\left[ 1 + \sum_{l\geq 1} a_{0l}(\ln |\la|)^{-l} +
\sum_{\bet_k>0}\sum_{-l_k}^{l_0} a_{kl}|\la|^{-\bet_k}(\ln |\la|)^l + O(|\la|^{-\infty})\right ]
\]

We denote by $F(\la)\,=\,\left[ 1 + \sum_{l\geq 1} a_{0l}(\ln |\la|)^{-l} +
\sum_{\bet_k>0}\sum_{-l_k}^{l_k} a_{kl}|\la|^{-\bet_k}(\ln |\la|)^l + O(|\la|^{-\infty})\right ]$

Taking the logarithmic derivative, we obtain

\[
-\frac{D'(-|\la|)}{D(-|\la|)}\,=\,2i\pi\xit'(-|\la|)\,=\,
\frac{\alp_0}{|\la|}\,+\,\frac{l_0}{|\la|\ln|\la|}+ \frac{F'(\la)}{F(\la)}.
\]

By inspection we find
\[
\frac{F'(\la)}{F(\la)}\,=\, \left\{
\begin{array}{lr}
O\left(|\la|^{-\bet_1-1}\right)& \mbox{regular case}\\
O\left(|\la|^{-1}(\ln|\la|)^{-2}\right) & \mbox{otherwise}.
\end{array}
\right .
\]

\begin{lem}\label{lem:xiasymp}
\begin{enumerate}
\item In the regular case, there exist three positive numbers $\alp_0,\,\beta_1$ and $M$ such that
the estimate
\begin{equation}
  \label{eq:xiasympnat}
\left | 2i\pi\xit'(-|\la|)-\frac{\alp_0}{|\la|}\right|\,
\leq\, M |\la|^{-\bet_1-1},
\end{equation}
holds for $\la$ large enough.
\item In the other cases, there exist two positive real numbers $\alp_0$ and $\bet_1,$ a non-negative integer
$l_0$ and a constant $M$ such that the estimate

\begin{equation}
  \label{eq:xiasymp}
\left |2i\pi \xit'(-|\la|)-\frac{\alp_0}{|\la|}-\frac{l_0}{|\la|\ln|\la|}\right|\,
\leq\, M \cdot|\la|^{-1}(\ln |\la|)^{-2}
\end{equation}
holds for $|\la|$ large enough.
\end{enumerate}
\end{lem}

In the regular case, for any $C>0,$ define $h_C(s)$ for $\Re(s)$ large enough by
\begin{equation}\label{eq:defhC}
h_C(s)\,=\,2i\pi\int_C^\infty \la^{-s}\xit'(-|\la|)\,d\la\,-\,\frac{\alp_0}{s}\exp(-s\ln(C)).
\end{equation}

The estimates of the previous lemma imply the following corollary. We restrict to the regular
case although similar statements are valid in the non-regular case (with extra
logarithmic singularities -see \cite{KLP}).
\begin{prop}
For a regular extension, the function $h_C$ extends to a holomorphic
function on $\left \{ \Re(s) \geq -\beta_1 \right \}.$
\end{prop}

\begin{proof}
We have
\begin{equation*}
\begin{split}
\int_C^\infty \la^{-s}2i\pi\xit'(-|\la|)d\la\,=\,{} & \int_C^\infty \la^{-s}\left[ 2i\pi\xit'(-|\la|)-\frac{\alp_0}{\la}\right] d\la\\
& +\,\int_C^{\infty} \la^{-s} \frac{\alp_0}{\la}\,d\la.
\end{split}
\end{equation*}
The second integral on the right-hand side is computed directly :
\[
\int_C^{\infty} \la^{-s} \frac{\alp_0}{\la}\,d\la\,=\,\frac{\alp_0}{s}\exp(-s\ln C),
\]
so that $h_C$ actually represents the first integral.
Lemma \ref{lem:xiasymp} then gives that $h_C$ extends to a holomorphic function on $\Re s >-\beta_1.$
\end{proof}
\subsection{Zeta-regularization}\label{sec:zetareg}

For any $A$ and any $C$ that is large enough, for any $\lat\in \Omega$
such that $\mbox{Re}(\lat)> A$ we choose a cut $c_{\lat}\subset \Omega$ that starts from
$-\infty+i0$ and that ends at $\lat.$ We may choose it in such a way that it
begins with the interval $(-\infty,-C].$

For any $\lat$ and any $s\in\C$, the function $\la \mapsto (\la-\lat)^{-s}$, which is well defined when
$\la-\lat$ is a positive real number, extends to a holomorphic function
on the complement of the cut $c_{\lat}$. Moreover, when $\la$ goes to the cut $c_{\lat}$ from
above or from below, we have the following jump condition
\[
\lim_{\la\downarrow c_{\lat}}\exp(-i\pi s)(\la-\lat)^{-s}\,=\,
\lim_{\la\uparrow c_{\lat}}\exp(i\pi s)(\la-\lat)^{-s}.
\]
For $\la$ on $c_{\lat}$, we define $(\la-\lat)_0^{-s}$ to be this common limit.

Let  $A^+$ be any number greater than $A$ that is neither an eigenvalue of $\dF$ nor of $\dL.$
Define a contour $\gamma$ that avoids $c_{\lat}$ and that consists in one part that encloses
the half-line $\{ x\geq A^+\}$ and then small circles that enclose the eigenvalues of $\dL$
and $\dF$ that are smaller than $A^+.$

For $\Re(s)>1$ we have
\begin{eqnarray*}
\zeta(s,\dL-\lat)&=& \frac{1}{2i\pi}\Tr \left (
\int_\gamma (\la-\lat)^{-s}\left( \dL-\la\right)^{-1}d\la\right),\\
&=&\frac{1}{2i\pi}\Tr\left (\int_{c_{\lat,\eps}} (\la-\lat)^{-s}\left( \dL-\la\right)^{-1}d\la\right),
\end{eqnarray*}
in which $c_{\lat,\eps}$ denotes the contour obtained by following $c_{\lat}$ at a distance $\eps.$
The second identity comes from Cauchy integral formula since, when
$\Re(s)>1$ the contribution of a large circle centered at $\lat$
tends to zero when the radius grows to infinity.

The same formulas are true for $\dF$ and using the fact that
$(\dL-\la)^{-s}$ and $(\dF-\la)^{-s}$ are trace class for $\mbox{Re} s>1$,
we can exchange the contour integration and the trace operation to obtain

\[
\zeta(s, \dL-\lat)-\zeta(s,\dF-\lat)\,=\,\frac{1}{2i\pi}\int_{c_{\lat,\eps}} (\la-\lat)^{-s}\Tr\left(
(\dL-\la)^{-1}-(\dF-\la)^{-1}\right)d\la
\]

Using prop. \ref{prop:difftrace} and the definition of $\xit$ we obtain
\[
\zeta(s, \dL-\lat)-\zeta(s,\dF-\lat)\,=\,\int_{c_{\la,\eps}} (\la-\lat)^{-s}
\xit'(\la)d\la.
\]

We rewrite the right-hand side in the following form :
$\zeta_1(s)+\zeta_2(s)$ where $\zeta_1$ corresponds to the part of the contour $c_{\la,\eps}$ that is in the
half-plane $\left\{ \mbox{Re} \la \leq -C\right \},$ and $\zeta_2$ is the remaining part of that contour.

The function $\zeta_2$ extends to an entire function of $s$ and for $\mbox{Re}(s)<1$ we may let $\eps$ go to
$0,$ giving
\[
\forall s,\,\Re(s)<1,~~~~~~\zeta_2(s)\,=\,2i\sin(\pi s)
\int_{-C}^\lat (\la-\lat)_0^{-s}\xit'(\la)d\la,
\]
where the integral is along the part of the cut $c_{\lat}$ that belongs to the half-plane
$\left \{ \mbox{Re}(\la) > -C \right \}.$

For $\zeta_1$ we may first let $\eps$ go to $0$ and  obtain :
\[
\zeta_1(s)\,=\, 2i\sin(\pi s)\int_{-\infty}^{-C} (\la-\lat)_0^{-s}\xit'(\la)d\la.
\]

We make a further reduction by using the following technical lemma.
\begin{lem}\label{lem:mulaasympt}
On $\C \times \{ |z|<1\},$ we define $\rho(s,z)\,=\,\left(1-z\right)^{-s}-1.$
For any $r<1,$ and any $R>0,$ the following holds for any $|z|\leq r,$ and any $|s|\leq R$

\begin{equation}
  \label{eq:mulaasympt}
  \left| \rho(s,z) \right |\,\leq\, \frac{\exp\left(\frac{Rr}{1-r}\right)}{1-r} \cdot |s|\cdot |z|.
\end{equation}
\end{lem}
\begin{proof}
We start from
\[
\rho(s,z) \,=\, \sum_{k\geq 1} \frac{(-s)^k \left[\ln(1-z)\right]^k}{k!}.
\]
By integration we have $|\ln(1-z)|\leq \frac{1}{1-r}|z|$ so that
\[
|\rho(s,z)| \leq \exp\left( \frac{|s||z|}{1-r}\right)-1\,=\,\int_0^{\frac{|s||z|}{1-r}}\exp(v)\,dv.
\]
The claim then follows.
\end{proof}

For $\mbox{Re}(\la)\leq -C,$ there exists some $r<1$ such that $\left|\frac{\lat}{\la}\right|\leq r$.
We can thus write
\[
(\la-\lat)^{-s}\,=\, \la^{-s}\left(1+\rho\left(s,\frac{\lat}{\la}\right)\right)
\]
for any $\la$ such that $\mbox{Re}(\la)\leq -C$ and $\la \notin (-\infty,-C).$

Fix some $R$, For $s$ such that $\mbox{Re}(s)>0$ and $|s|\leq R,$ using the bound
in Lemma \ref{lem:mulaasympt} we may let $\eps$ go to zero and write

\[
\zeta_1(s)\,=\, 2i\sin(\pi s) \int_{-\infty}^{-C} |\la|^{-s}\xit'(\la)\,d\la
+\, 2i\sin(\pi s)\tilde{R}_C(s,\lat)\,
\]
where
\[
\tilde{R}_C(s,\lat)\,=\, \int_{-\infty}^{-C} |\la|^{-s}\xit'(\la)\rho\left(s,\frac{\lat}{\la}\right) d\la.
\]

Using Lemma \ref{lem:mulaasympt} and  Lemma \ref{lem:xiasymp} we find that, for any
extension (regular or not) ${\tilde{R}}_C( \cdot , \la)$
can be analytically continued to $\Re(s)> -1$, and that ${\tilde{R}}_C(0)=0.$

Adding up $\zeta_1$ and $\zeta_2$ we obtain the following proposition.

\begin{prop}
For any extension, the function $R_C(s,\lat)$ which is defined for $s$ large by
\begin{equation*}
R_C(s, \lat )= \zeta(s,\dL-\lat)-\zeta(s,\dF-\lat)- 2i \sin(\pi s)\int_{-\infty}^{-C} |\la|^{-s}\xit'(\la)\,d\la -\zeta_2(s)
\end{equation*}
can be analytically continued to $\Re(s)\,>\,-1$ and $R_C(s,\lat)$ vanishes at least at second order
at $s=0.$
\end{prop}

\begin{proof}
By inspection and using the definitions of the different functions that appear in the
expression of $R_C$ we find that
\[
R_C(s,\lat) \,=\, 2i\sin(\pi s)\tilde{R}_C(s,\lat).\]
Using the bounds given by Lemmas \ref{lem:xiasymp} and \ref{lem:mulaasympt} we find a constant $\tilde{C}$ such that
\[
\forall \lambda<-C,~ ~ \left | |\la|^{-s}\xit'(\la)\rho\left(s,\frac{\lat}{\la}\right) \right| \,\leq\, \tilde{C} |s|\cdot|\lambda|^{-\mathrm{Re}(s)-2},
\]
where $\tilde{C}$ depends on $C, \tilde{\lambda}$ and is uniform for $|s|\leq R.$ The claim follows
\end{proof}

In particular, in the regular case, we obtain the following corollary
(compare with \cite{Mooers})

\begin{coro}
If $L$ defines a regular extension then $(s-1)\zeta(s,\dL-\lat)$ extends to a
holomorphic function on $\Re(s)> -\beta_1.$
\end{coro}

\begin{proof}
The zeta regularization of the Friedrichs extension is well-known and well
studied starting from the small-time asymptotics of the heat kernel (obtained
for instance from \cite{Cheeger}). The function $(s-1)\zeta(\dF-\lat)$ is thus
known to extends holomorphically to $\C$ (see \cite{AurSal, BDK, KLP, KokKor}).
Moreover the preceding proposition yields that
\[
 (s-1)\zeta(s,\dL-\lat)\,=\,(s-1)\cdot \left[ \zeta(s,\dF-\lat)\,+\,\frac{\sin(\pi s)}{\pi}\left( h_C(s) \,+\,
\frac{\alp_0}{s}\exp(-s\ln C)\right) \, +\, \zeta_2(s)\,+\,R_C(s,\tilde{\lambda})\right].
\]

The statement thus follows by examining the analytic continuation of each
individual term.
\end{proof}

\begin{remk}
By evaluating everything at $s=0$ we obtain
\[
\zeta(0,\dL-\lat)\,=\,\zeta(0,\dF-\lat)+\alp_0.
\]
\end{remk}

In the regular case, we can thus define the regularized zeta determinant
by the usual formula
\[
{\det}_\zeta(\dL-\lat)\,=\, \exp\left(-\zeta'(0,\dL-\lat)\right),
\]
and we obtain the following theorem.

\begin{theorem}\label{thm:BKreg}
Let $L$ define a regular extension and set $\Gamma$ to be
\begin{equation}
  \label{eq:defGamma}
  \Gamma = \lim_{\la\rightarrow \infty} \ln\left(D(-|\la|)\right)-\alp_0\ln(-|\la|).
\end{equation}
Then, for any $\lat\in\Omega$ we have
\begin{equation}\label{eq:BKreg}
{\det}_\zeta(\dL-\lat)\,=\,e^{-\Gamma}\cdot D(\lat)\cdot {\det}_\zeta(\dF-\lat).
\end{equation}
\end{theorem}
\begin{proof}
According to the preceding proposition, we have
\[
\zeta'(0,\dL-\lat)-\zeta'(0,\dF-\lat)\,=\,
\zeta_2'(0)+ h_C(0)-\alp_0\ln(C).
\]
We compute
\[
\zeta_2'(0)\,=\,2i\pi \left[ \xit(\lat)-\xit(-C)\right].
\]
Combining the two we find
\begin{equation*}
\begin{split}
\zeta'(0,\dL-\lat)-\zeta'(0,\dF-\lat)\,&=\,
2i\pi \xit(\lat)-2i\pi \xit(-C)+h_C(0)-\alp_0\ln(C)\\
&=\, 2i\pi\xit(\lat)+\ln\left(D(-C)\right)-\alp_0\ln(C)+h_C(0)
\end{split}
\end{equation*}

This implies the result with $\Gamma$ replaced by the quantity $\ln(D(-C))-\alp_0\ln(C)+h_C(0)$
(which proves in particular that the latter doesn't depend on $C$ large enough).
When we let $C$ go to infinity, on the one hand
$\ln(D(-C))-\alp_0\ln(C)$ goes to $\Gamma,$ and on the other hand, since
\[
h_C(0)\,=\, \int_C^\infty \left(2i\pi\xit'(-|\la|)-\frac{\alp_0}{\la}\right) \,d\la
\]
and the function inside the integral is $L^1,$
$h_C(0)$ goes to $0.$ This finishes proving the theorem.
\end{proof}

\begin{remk}
As soon as $h_C$ allows the definition of the relative zeta determinant of
$\dL-\lat$ and $\dF-\lat,$ then, using theorem \ref{thm:BKreg} and differentiating with
respect to $\la,$ we recover a well-known fact of this theory :
\[
\partial_\lat \left(\ln\det(\dL-\lat)-\ln\det(\dF-\lat)\right)\,=\, 2i\pi\xit'(\lat).
\]
(compare with \cite{Forman, HassZeld, Carron})
\end{remk}

\begin{remk}
For non-regular extensions, it is still possible to analytically continue
$\zeta$ to $\Re s>0$ and to define a zeta-regularized determinant by picking some coefficient
in the asymptotic expansion of $\zeta(s,\dL-\lat)$ near $0$ (see \cite{KLP}). Note however, that
the limit $\tilde{\lambda}\rightarrow 0$ will be problematic.
\end{remk}

\subsection{Proof of theorem \ref{thm:intromain}}
In order to get the theorem of the introduction, we now let $\lat$ go to $0$.
We thus modify the zeta-regularized determinant in order to exclude the
eigenvalue $0.$ Define by $\delta_L$ (resp. $\delta_F$) the dimensions
of $\ker(\dL)$ (resp. $\ker(\dF)$). Equation \refeq{eq:difftraceD} implies that
$0$ is a pole of $\frac{D'}{D}$ with residue $d:=\delta_L-\delta_F$ so that we can define
$D^*(0)$
\[
D^*(0):=\lim_{\la\rightarrow 0}D(\la)(-\la)^{-(\delta_L-\delta_F)}.
\]

On the other hand, we define the modified zeta function by
\[
\zeta^*(s,\dF-\lat)\,=\,\zeta(s,\dF-\lat)-\delta_F (-\lat)^{-s}
\]
and the corresponding modified determinant.

\begin{defn}\label{def:moddet}
Let $L$ be defining a regular extension (or $L=F$),
the modified zeta determinant of $\Delta_L$ is defined by
\[
{\det}^*_{\zeta}(\dL)\,=\,\lim_{\tilde{\la} \rightarrow 0} (-\lat)^{-\delta_L}
{\det}_{\zeta}(\dL-\lat)
\]
\end{defn}

Using this definition for $\dL$ and $\dF$, and plugging into \refeq{eq:BKreg}, the powers of
$-\lat$ cancel out and we may let $\lat$ go to zero. We thus obtain the theorem
in the introduction (Thm. \ref{thm:intromain}).

When $d=0$, the prefactor $D^*(0)$ may be computed using the method of section \ref{sec:S0}.
When $d>0$ then this method has to be refined to compute more terms in the Taylor expansion
of $S(\la)$ at $\la=0.$ In the following example, we will pay special attention
to addressing the question of the kernel of $P+QS(0).$

\subsection{On the Euclidean sphere with one $4\pi$ and  six $\pi$ singularities}
We consider the Euclidean sphere with six $\pi$ singularities and one $4\pi$ conical
point.
We define
\[
A^{\pm}\,=\,\left(\begin{array}{c}
a_0^\pm \\
a_{-\und}^\pm \\
a_{+\und}^\pm \\
{{\tilde{A}}^\pm}
\end{array}
\right )
\]
where $a_i^\pm,~ i=-\und,\,0,\,\und$ correspond to the $4\pi$ singularity and
${\tilde{A}}^{\pm}$ are the coefficients corresponding to the remaining  six $\pi$ singularities.
Recall that for each of the latter there are only two coefficients $a_0^\pm.$

A regular extension thus relates only the coefficients $a_{\pm \und}^\pm$

We define $P_{\theta}$ and $Q_{\theta}$ by
\[
P_{\theta}:=\left( \begin{array}{ccc}
1 & 0 & 0 \\
0& \cos\theta I_2 & 0\\
0& 0& I
\end{array}
\right) ~,~
Q_{\theta}:=\left( \begin{array}{ccc}
0 & 0 & 0 \\
0& \sin\theta I_2 & 0\\
0& 0& 0
\end{array}
\right).
\]

This choice defines a regular self-adjoint extension (which is, moreover invariant under complex
conjugation). We have
\[
D(\la)=\det(P+QS(\la))=\det( \cos\theta I_2 +\sin\theta \tilde{S}(\la)),
\]

where $\tilde S$ is the $2\times 2$ matrix obtained from $S$ by erasing the first row and column
(that correspond to $a_0^\pm$) and all the rows and columns corresponding to ${\tilde{A}}^\pm$

According to proposition \ref{prop:Sasymp}, when $\theta\neq 0, \pi$, the asymptotic expansion of $D$ is given by
\begin{equation*}
\begin{split}
\ln D(-|\la|)&= 2|\nu|\ln( |\la|) +
\ln\Big( \left[ \frac{\Gamma(1-|\nu|)}{2^{2|\nu|}\Gamma(1+|\nu|)}\sin \theta\right]^2\Big)+O(1),\\
& = \ln( |\la|) + \ln \left[ \sin \theta\right]^2\,+O(1),
\end{split}
\end{equation*}
since $|\nu|=\und.$

Finally, we obtain that, for any $\theta\neq 0~,\pi$ such that $-\mbox{cotan}(\theta)$ isn't an eigenvalue of
$\tilde{S}(0)$ the following holds :
\begin{equation*}
{\det}^*_{\zeta}(\dL)=
\frac{\det(\cos\theta I_2 +\sin\theta \tilde{S}(0))}{\sin^2\theta} \cdot {\det}^*_\zeta(\dF).
\end{equation*}

\end{document}